\documentclass[10pt]{amsart}

\usepackage{amsmath,amsfonts,amsthm,amssymb}
\usepackage[utf8]{inputenc}
\usepackage[pdfborder={0 0 0 [0 0 ]}]{hyperref}
\usepackage{graphicx}
\usepackage{color}
\usepackage[ps,all,arc,rotate,cmtip]{xy}
\usepackage{verbatim}
\usepackage[alphabetic]{amsrefs}
\usepackage{pgf,tikz}
\usetikzlibrary{arrows}
\usepackage{caption}
\usepackage{enumitem}
\usepackage[capitalise]{cleveref}

\long\def\symbolfootnote[#1]#2{\begingroup 
\def\thefootnote{\fnsymbol{footnote}}\footnote[#1]{#2}\endgroup}
\numberwithin{equation}{section}

\newtheorem*{theorem*}{Theorem}
\newtheorem{theorem}{Theorem}[section]

\newtheorem{lem}[theorem]{Lemma}
\newtheorem{thm}[theorem]{Theorem}

\newtheorem{prop}[theorem]{Proposition}

\newtheorem{cor}[theorem]{Corollary}

\theoremstyle{definition}

\newtheorem{rmk}[theorem]{Remark}
\newtheorem{remark}[theorem]{Remark}

\newtheorem{defin}[theorem]{Definition}
\newtheorem{dfn}[theorem]{Definition}
\newtheorem{constr}[theorem]{Construction}
\newtheorem{ex}[theorem]{Example}

\newtheoremstyle{maintheorem}{}{}{\itshape}{}{\bfseries}{}{.5em}{#1 \!\thmnote{#3}.}
\theoremstyle{maintheorem}
\newtheorem*{mainthm}{Theorem}

\newcommand{\R}{\mathbb{R}}
\newcommand{\Z}{\mathbb{Z}}
\renewcommand{\H}{\mathcal{H}}

\newcommand{\C}{\mathcal{C}}

\newcommand{\B}{\mathcal{B}}
\newcommand{\F}{\mathcal{F}}

\newcommand{\norm}{\mathfrak{N}}

\newcommand{\U}{\mathcal{U}}
\renewcommand{\P}{\mathcal{P}}

\newcommand{\Hom}{\operatorname{Hom}}

\newcommand{\Map}{\operatorname{Map}}

\newcommand{\ab}{{\operatorname{ab}}}
\newcommand{\Pic}{{\operatorname{Pic}}}

\newcommand{\hull}{{\operatorname{hull}}}
\newcommand{\sym}{{\operatorname{sym}}}

\newcommand{\rank}{{\operatorname{rank}}}
\newcommand{\Sh}{{\operatorname{Sh}}}
\newcommand{\pr}{{\operatorname{pr}}}
\newcommand{\codim}{{\operatorname{codim}}}

\newcommand{\amp}{\operatorname{amp}}

\newcommand{\ZZ}{\mathcal{Z}}
\newcommand{\zz}{\boldsymbol{z}}
\renewcommand{\subset}{\subseteq}

\newcounter{flocomments}

\renewcommand{\leq}{\leqslant}
\renewcommand{\geq}{\geqslant}

\def\im{\mathrm{im}}

\def\id{\mathrm{id}}

\def\s-{\smallsetminus}

\def\supp{\mathrm{supp} \,}

\def\P{\mathcal{P}}

\renewcommand{\phi}{\varphi}

\begin{document}
	
	\title{The integral polytope group}
	\author{Florian Funke}
	\address{Mathematisches Institut der Universit\"{a}t Bonn, Endenicher Allee 60, 53115 Bonn, Germany}
	\email{ffunke@math.uni-bonn.de}
	
	\begin{abstract}
		We show that the Grothendieck group associated to integral polytopes in $\R^n$ is free-abelian by providing an explicit basis. Moreover, we identify the involution on this polytope group given by reflection about the origin as a sum of Euler characteristic type. We also compute the kernel of the norm map sending a polytope to its induced seminorm on the dual of $\R^n$.
	\end{abstract}
	
	\maketitle

\section{Introduction}

\subsection{Motivation}
The set of polytopes in a real finite-dimensional vector space $V$ forms a commutative monoid under pointwise addition, also called \emph{Minkowski sum}, and we denote its Grothendieck group by $\P(V)$. Given a finitely generated free-abelian group $H$, a polytope in $V_H = H\otimes_\Z\R$ is \emph{integral} if all of its vertices lie in $H$. This condition determines a subgroup $\P(H)\subseteq \P(V_H)$ called the \emph{integral polytope group}. Identifying polytopes which are translates of each other produces a quotient $\P_T(H)$ of $\P(H)$. 

Our motivation for studying these groups is that they are the places where coarse versions of algebraic invariants from low-dimensional topology take values in. These are obtained by (variations of) the following procedure: Let $R$ be a ring and $G$ a group, and denote by $\pr\colon G\to H_1(G)_f$ the projection onto the free part of the first integral homology $H_1(G)$ of $G$. Let $R*G$ be a crossed product ring (e.g., the usual group ring or a twisted Laurent polynomial ring, see \cite[Section 10.3.2]{Lueck2002}) and assume that it does not contain zero-divisors. Then the map
\[ P\colon R*G \s- \{0\}\to \P(H_1(G)_f), \;\; x\mapsto P(x) = \text{convex hull}(\pr(\supp(x)))\]
satisfies $P(x\cdot y) = P(x) + P(y)$. If $R*G$ satisfies the Ore condition (see, e.g., \cite[Definition 8.14]{Lueck2002}) with respect to $T = R*G\setminus\{0\}$, then $P$ passes to the units of the localization $D = T^{-1}(R*G)$ and induces a group homomorphism
\begin{equation} \label{polytope homo}
P\colon D^\times_\ab\to \P(H_1(G)_f), \;\; b^{-1}a\mapsto P(a) - P(b).
\end{equation}
Since $D$ is a skew-field, there is an isomorphism $K_1(D)\cong D^\times_\ab$ given by the Dieudonn\'e determinant \cite[Corollary 2.2.6]{Rosenberg1994}. Thus we may push forward any invariant with values in $K_1(D)$ to a polytope invariant with values in $\P(H_1(G)_f)$. 

This procedure can be applied to twisted and higher-order Alexander polynomials
\cite{McMullen2002, COT2003, Cochran2004, Harvey05, Friedl2007, FriedlHarvey2007, FriedlVidussi2010}. It has most recently been examined and applied by Friedl-L\"uck \cite{FriedlLueck2015, FriedlLueck2015b} to their universal $L^2$-torsion in order to construct the $L^2$-torsion polytope, see also \cite{FunkeKielak2016}. On the other hand, a thorough understanding of the polytope group itself is just beginning to emerge \cite{FriedlLueck2015b, ChaFriedlFunke2015}.

An integral polytope $P\subseteq V_H$ induces a seminorm on $\Hom_\Z(H,\R)$ by setting
\[ \|\phi\|_P = \max\{\phi(p)-\phi(q)\mid p,q\in P\}, \]
and the equation $\|\phi\|_{P+Q} = \|\phi\|_P + \|\phi\|_Q$ is immediate. The set of (set-theoretic) maps $\Map(\Hom_\Z(H, \R), \R)$ is a group under pointwise addition, and we obtain a group homomorphism
\begin{equation}\label{seminorm map}
	\norm\colon\P(H)\to \Map(\Hom_\Z(H, \R), \R), \; P-Q \mapsto \|\cdot\|_P-\|\cdot\|_Q.
\end{equation}
McMullen's Alexander norm, Harvey's higher-order Alexander norms, and the Thurston norm of a compact connected orientable $3$-manifold $M$ with empty or toroidal boundary are in the image of $\norm\circ P$ for suitable varying skew-fields $D$, see \cite{FriedlHarvey2007} and \cite{FriedlLueck2015b}. While the Alexander norm and its higher-order friends are also defined for HNN extensions of free groups, \cite{FriedlLueck2015b} makes way for an analogue of the Thurston norm for these HNN extensions.

One motivation for a better understanding of the integral polytope group is to lift the known inequalities between these seminorms \cite{McMullen2002, Harvey06} to the polytope classes in $\P(H_1(G)_f)$ inducing them. Such a conceptual reason might help to put the newly defined Thurston norm of HNN extensions of free groups into a bigger picture. This strategy was exploited in \cite{FunkeKielak2016} to show that if the free base group is of rank $2$, then this Thurston norm satisfies inequalities with the higher Alexander norms which are completely analogous to the $3$-manifold setting.

\subsection{Connections to toric geometry}

A \emph{toric variety} is an irreducible algebraic variety $X$ containing a torus $T_N \cong (\mathbb{C}^n)^*$ as a Zariski open subset such that the action of $T_N$ on itself extends to an action on $X$, compare \cite[Definition 3.1.1]{Cox2011}. The interplay between toric varieties on the one hand and polytopes on the other hand is well-established, see \cite{Cox2011, Bruns2009, Gelfand2008, Sturmfels1996}. The standard construction producing a toric variety from a full-dimensional integral polytope $P\in \P(\Z^n)$, or more generally from a \emph{fan}, is one example of this connection, see \cite[Definition 2.3.13]{Cox2011} or \cite[Definitiion 4.2]{Gelfand2008}.

Recall that the \emph{Picard group} $\Pic(X)$ of a projective variety $X$ is defined as the group of isomorphism classes of line bundles over $X$ with multiplicaton induced by the tensor product. A line bundle $L$ is \emph{ample} if, roughly speaking, some power $L^k$ admits enough global sections so as to construct an embedding $X\to \mathbb{P}^N$. Ample line bundles determine a subgroup $\Pic_{\amp}(X)\subset \Pic(X)$. If $X$ is a normal projective toric variety, then $X$ is induced by a fan $\mathcal{F}$ as mentioned above, see \cite[Theorem 4.3 (a)]{Gelfand2008}. Thus \cite[Theorem 10.11]{Bruns2009} implies that there is a homomorphism
\[ \Pic_{\amp}(X) \to \P_T(H),\]
where $H$ is a free-abelian group whose rank is equal to the dimension of the Zariski open torus of $X$. The image of this homomorphism is contained in the set of polytopes with normal fan equal to $\mathcal{F}$. We hope that this connection sparks further analysis of the integral polytope group and its subgroups.

\subsection{Results}

It is proved in \cite[Lemma 3.8]{FriedlLueck2015b} that $\P(H)$ embeds into a countably infinite product of infinite cyclic groups. Therefore, a theorem of Specker \cite{Specker1950} states that $\P(H)$ is an infinitely generated free-abelian group for every finitely generated free-abelian group $H$. While this conclusion is interesting, it does not provide any geometric insight. We fill this gap by providing an explicit, geometrically tangible basis for $\P(H)$. More explicitly, given a subgroup $G\subseteq H$, let $\P_T^m(G)$ denote the subgroup of $\P_T(G)$ generated by polytopes of dimension at most $m$. Then our main result is as follows:

\begin{mainthm}[\ref{main theorem} (Basis for the integral polytope group)]
	Let $H$ be a finitely generated free-abelian group. Then there are sets $\B_1\subseteq \B_2 \subseteq  ... \subseteq\B_n\subseteq\P_T(H)$ such that $\B_m\setminus \B_{m-1}$ contains only polytopes of dimension $m$ and $\B_m\cap \P_T(G)$ is a basis for $\P^m_T(G)$ for every pure subgroup $G\subseteq H$ and $1\leq m\leq n$.	In particular, $\B_n$ is a basis for $\P_T(H)$.
\end{mainthm}

The methods used in its construction also produce a basis for the real vector space $\P(V)$. Moreover, since there is a split short exact sequence
\[ 0 \to H \to \P(H) \to \P_T(H) \to 0,\]
the above theorem also gives a basis for $\P(H)$. 

Secondly, we show that the natural involution $*:\P(V)\to \P(V)$ on the polytope group given by reflection about the origin has a description in terms of the faces of a polytope.

\begin{mainthm}[\ref{main theorem 2} (Involution as face Euler characteristic)]
	Let $V$ be a finite-dimensional real vector space and $P\subseteq V$ be a polytope. Then we have in $\P(V)$
	\[ *P = -\sum_{F\in\F(P)} (-1)^{\dim(F)}\cdot F,\]
	where $\F(P)$ denotes the set of faces of $P$ (including $P$ itself).
\end{mainthm}

Finally, the main theorem of \cite{ChaFriedlFunke2015} states that 
\[\ker \big(\id-*\colon \P(H)\to \P(H)\big) = \im\big(\id + *\colon \P(H)\to \P(H)\big).\]
We prove here the following dual statement.

\begin{mainthm}[\ref{main theorem 3}]
	We have
	\[\ker \big(\id+*\colon \P(H)\to \P(H)\big) = \im\big(\id - *\colon \P(H)\to \P(H)\big)\]
	and 
	\[\ker \big(\id+*\colon \P_T(H)\to \P_T(H)\big) 
	=  \im\big(\id-*\colon \P_T(H)\to\P_T(H)\big).\]
\end{mainthm}

It is well-known that two integral polytopes $P$ and $Q$ induce the same seminorm on $\Hom_\Z(H, \R)$ if and only if $P+*P = Q+*Q$. By the latter theorem, this is equivalent to the existence of an integral polytope $R$ such that $P +*R = Q + R$, thus directly relating $P$ and $Q$; see \cref{lem:equivalence norm}. \cref{main theorem 3} is used in \cite[Proposition 6.3]{Funke2017} to put restrictions on the possible shape of the $L^2$-torsion polytope of amenable groups.

\subsection*{Acknowledgements} 

The author was supported by GRK 1150 `Homotopy and Cohomology' funded by the DFG, the Max Planck Institute for Mathematics, and the Deutsche Telekom Stiftung. We thank Stefan Friedl, Fabian Henneke, Dawid Kielak, and Wolfgang L\"uck for many fruitful discussions. We also thank the referee for carefully reading our work and for pointing out important connections to toric geometry.
\newpage
\vspace{7mm}
\tableofcontents

\vspace{3mm}
\section{Preliminaries on the polytope group}

Let $V$ be a finite-dimensional real vector space. A \emph{polytope in $V$} is a subset $P\subseteq V$ which is the convex hull of finitely many points. The \emph{dimension} of $P$ is the dimension of the smallest affine subspace $U\subseteq V$ containing $P$. Its \emph{boundary} $\partial P\subseteq P$ is the boundary of $P$ inside $U$. 

Given two polytopes $P$ and $Q$ in $V$, their \emph{Minkowski sum} is defined as
\[ P + Q = \{ p + q \in V \mid p\in P, q\in Q\}.\]
The Minkowski sum is \emph{cancellative} in the sense that $P_1 + Q = P_2 + Q$ implies $P_1 = P_2$, see e.g. \cite[Lemma 3.1.8]{Schneider1993}. It turns the set of polytopes in $V$ into a commutative monoid. The \emph{polytope group of $V$}, denoted by $\P(V)$, is defined as the Grothendieck group of this monoid, i.e.,  elements in $\P(V)$ are formal differences $P-Q$, subject to the equality $P-Q = P'-Q'$ if and only if $P+Q' = P'+Q$ holds as subsets in $V$. The image of a polytope $P$ in $\P(V)$ will still be denoted by $P$ in order to avoid an overload of notation. 

Let $H$ be a finitely generated free-abelian group. A polytope $P$ in $V_H =H\otimes_\Z \R$ is \emph{integral} if it is the convex hull of finitely many points in $H$ considered as a lattice inside $V_H$. In this case we sometimes say that $P$ is a polytope in $H$. The set of integral polytopes forms a submonoid of the monoid of polytopes in $V_H$, and its Grothendieck group will be denoted by $\P(H)$. 

There is a map of real vector spaces 
\[V\to\P(V), \; v\mapsto \{v\}\] 
and we denote the cokernel of this map by $\P_T(V)$, where the subscript $T$ stands for \emph{translation}. We define $\P_T(H)$ similarly. Thus two integral polytopes in $V_H$ determine the same class in $\P_T(H)$ if and only if they are translates of each other. There are natural inclusions $\P(H)\to\P(V_H)$ and $\P_T(H)\to\P_T(V_H)$. Moreover, a group homomorphism $f\colon H\to H'$ of finitely generated free-abelian groups induces morphisms 
\begin{gather*}
\P(f)\colon\P(H)\to\P(H');\\
\P_T(f)\colon\P_T(H)\to\P_T(H').
\end{gather*}
by sending the class of a polytope $P$ to the class of the polytope $(\R\otimes f)(P)$. If $f$ is injective, then both $\P(f)$ and $\P_T(f)$ are easily seen to be injective as well. Thus if $G\subseteq H$ is a subgroup, then we will always view $\P(G)$ (respectively $\P_T(G)$) as a subgroup of $\P(H)$ (respectively $\P_T(H)$).

Given a polytope $P\subseteq V$, we denote by $*P = \{-p\in V\mid p\in P\}$ the polytope obtained from $P$ by reflection about the origin. We obtain an involution 
\[*: \P(V)\to\P(V), \; P - Q \mapsto *P - *Q\]
which induces involutions on $\P_T(V)$, $\P(H)$, and $\P_T(H)$.

\begin{ex}\label{ex:polytope dim 1}
	Integral polytopes in $V_\Z = \R$ are just intervals $[m,n]\subseteq\R$ starting and ending at integral points. Thus we have $\P(\Z) \cong \Z^2$, where an explicit isomorphism is given by sending the class $[m,n]$ to $(m, n-m)$. Under this isomorphism, the involution corresponds to $*(k,l) = (-l-k, l)$. Similarly, $\P_T(\Z) \cong\Z$, where an explicit isomorphism is given by sending the element $[m,n]$ to $n-m$. The involution $*$ on $\P_T(\Z)$ is the identity.
\end{ex}

\vspace{3mm}
\section{Geometric tools}

In this section we will review a few basics of polytope theory and build up the geometric language used in the construction of a basis for $\P(H)$.

Throughout, we let $\zz\in\R^n$ denote the point $(0,...,0,1)$ and $\ZZ\subseteq\R^n$ the $1$-dimensional polytope whose vertices are $0$ and $\zz$. We denote by $\zz^\perp$ the orthogonal complement of $\zz$ with respect to the standard inner product.

\subsection{Face maps}

Let $V$ be a finite-dimensional real vector space. A \emph{hyperplane} $H\subseteq V$ is a subset of the form $H = \{x\in V\mid \phi(x) = c\}$ for some $\phi\in\Hom(V,\R)$ and $c\in\R$. A hyperplane in $\R^n$ is \emph{horizontal} if it is a translate of $\zz^\perp$, and a polytope in $\R^n$ is \emph{horizontal} if it lies in a horizontal hyperplane.

\begin{defin}
	Given $\phi\in\Hom(V,\R)$ and a polytope $P\subseteq V$, then we let 
	\[ F_\phi(P) = \{p\in P\mid \phi(p) = \max\{ \phi(q)\mid q\in P\} \}\]
	and call it the \emph{face in $\phi$-direction}. A subset $F\subseteq P$ is called a \emph{face} if $F_\phi(P) = F$ for some $\phi\in\Hom(V,\R)$. A face is a polytope in its own right, and its \emph{codimension} is defined as 
	\[ \codim(F\subset P) = \dim(P) - \dim(F).\] 
	A face of codimension $1$ will be referred to as a \emph{facet}. The set of faces of $P$ will be denoted by $\F(P)$. Note that $\F(P)$ includes the codimension $0$ face $P$.
	
	It is easy to see that $F_\phi(P+Q) = F_\phi(P)+F_\phi(Q)$ for any two polytopes. This implies that we obtain an induced \emph{face map}
	\begin{equation} \label{face map}
	F_\phi\colon \P(V)\to \P(V), \; P\mapsto F_\phi(P)
	\end{equation}
	which is a group homomorphism.
	
	If $P\subseteq \R^n$ is a polytope of dimension $n$ and $F$ is a facet, then there is up to positive scalar a unique $\phi\in\Hom(\R^n,\R)$ with $F_\phi(P) = F$. The face $F$ will be called \emph{bottom, vertical}, or \emph{top} face depending on whether $\phi(\zz) < 0$, $\phi(\zz) = 0$, or $\phi(\zz)>0$. 
\end{defin} 

A face $F$ of $P$ is a bottom (resp. vertical, top) face if and only if the face $*F$ of $*P$ is a top (resp. vertical, bottom) face.

\subsection{Height and shadow maps} \label{sec:shadow maps}

Given a subset $S\subseteq \R^n$, the convex hull of $S$ will be denoted by $\hull(S)$. Moreover, we call
\[h(S) = \inf\{ x_n\mid x\in S \}\]
the \emph{height} of $S$. It is obvious that $h(S+T) = h(S) + h(T)$, so that we get an induced homomorphism
\[ h\colon \P(\R^n) \to \R,\; P\mapsto h(P).\]

Given some $h\in\R$ consider the map 
\begin{equation}\label{compression}
	c_h\colon \R^n\to\R^n,\; (x_1,..., x_n)\mapsto (x_1, ..., x_{n-1}, h),
\end{equation}
which we can think of as compressing the vector space to a horizontal hyperplane.

\begin{defin}\label{def:shadow map}
	The \emph{shadow} of a polytope $P\subseteq\R^n$ is defined as
	\[ \Sh(P) = \hull(P \cup c_{h(P)}(P)).\] 
\end{defin}

The shadow of a (integral) polytope is again a (integral) polytope and allows us to increase the dimension in a simple controlled fashion. It comes perhaps not as a surprise that it will be our main tool to build a basis for $\P(\Z^n)$ out of one for $\P(\Z^{n-1})$. In this process, it is crucial that taking shadows preserves the algebraic structure, as shown by the next lemma.

\begin{lem}\label{shadow lemma}
	Given two polytopes $P, Q\subseteq\R^n$, we have 
	\[ \Sh(P + Q) = \Sh(P) + \Sh(Q)\]
	and hence we obtain a well-defined group homomorphism
	\[\Sh \colon \P(\R^n)\to \P(\R^n), \; P\mapsto \Sh(P)\]
	called \emph{shadow map}.
\end{lem}
\begin{proof}	
	It is well-known that for any subsets $S,T\subseteq \R^n$ we have 
	\[ \hull(S + T) = \hull(S) + \hull(T),\]
	see, e.g., \cite[Theorem 1.1.2]{Schneider1993}. Hence it suffices to show that 
	\begin{equation}\label{shadow equation}
		\hull((P +Q) \cup c_{h(P+Q)}(P + Q)) = \hull((P \cup c_{h(P)}(P)) + (Q \cup c_{h(Q)}(Q))).
	\end{equation}
	
	Note that $h(P+Q) = h(P)+h(Q)$. The inclusion $\subseteq$ already follows from the inclusion of the underlying sets
	\begin{equation} \label{shadow inclusion}
	(P +Q) \cup c_{h(P+Q)}(P + Q) \subseteq (P \cup c_{h(P)}(P)) + (Q \cup c_{h(Q)}(Q)).
	\end{equation}
	
	For the inclusion $\supseteq$, let $p\in P \cup c_{h(P)}(P)$ and $q\in Q\cup c_{h(Q)}(Q)$, and we will show that $p+q$ is contained in the left-hand side of (\ref{shadow equation}). This is obvious if $(p,q)\in P\times Q$ or $(p,q)\in c_{h(P)}(P)\times c_{h(Q)}(Q)$. 
	
	Let us now assume that $p\in P$ and $q\in c_{h(Q)}(Q)$. Write $q = c_{h(Q)}(q')$ for some $q'\in Q$. Then $p+q$ lies on the convex hull of the points $p +q'$ and $c_{h(P)}(p) + q = c_{h(P+Q)}(p+q')$. By inclusion (\ref{shadow inclusion}), these latter points lie in 
	\[\hull((P \cup c_{h(P)}(P)) + (Q \cup c_{h(Q)}(Q)))\] 
	and hence so does $p+q$. The case $p\in c_{h(P)}(P)$ and $q\in Q$ is completely analogous.
\end{proof}

It is straightforward to see that $\Sh\colon \P(\R^n)\to \P(\R^n)$ induces shadow maps on $\P(\Z^n)$, $\P_T(\R^n)$, and $\P_T(\Z^n)$.

\begin{remark}\label{rem:upper maps}
	The choice of $\min$ instead of $\max$ in \cref{def:shadow map} is of course arbitrary. Completely analogously, we may define an \emph{upper height map}
	\[ h^+\colon \P(\R^n)\to\R, \;\; P\mapsto \max\{ x_n\mid x\in P \}\]
	and an \emph{upper shadow map}
	\[ \Sh^+\colon \P(\R^n)\to \P(\R^n), \;\; P\mapsto\hull(P \cup c_{h^+(P)}(P)).\]
	Then the equations
	\[h^+(*P) = -h(P)\;\;\text{ and }\;\; \Sh^+(*P) = *\Sh(P)\]
	are easy to verify.
\end{remark}

\subsection{Partition relations}

In this section we extend results from \cite[Section 3.2]{ChaFriedlFunke2015} on how to manipulate Minkowski sums geometrically. Take a hyperplane $H = \{x\in V\mid \phi(x) = c\}$ for some $\phi\in\Hom(V,\R)$ and $c\in\R$. Given a polytope $P\subseteq V$, the two \emph{halves} of $P$ with respect to $H$ are defined as
\begin{align*} 
	P_+  &= \{ p\in P\mid \phi(p) \geq c\}\\
	P_-  &= \{ p\in P\mid \phi(p) \leq c\}.
\end{align*}

Of course, $\phi$ is unique only up to scalar and so the subscripts in the notation are arbitrary. Note that a half is either empty, a face of $P$ or a subpolytope of codimension $0$.

The geometric process of cutting $P$ along $H$ into the two halves $P_+$ and $P_-$ yields the following algebraic equation.

\begin{lem}[Cutting relation]\label{cutting along a hyperplane}
	Let $P, P_+,P_-, H\subseteq V$ be as above. Then
	\[ P_+ + P_- = P + (P\cap H) .\]
\end{lem}
\begin{proof}
	This is proved in \cite[Lemma 3.2]{ChaFriedlFunke2015}.
\end{proof}

In our application it is necessary to cut a polytope along more complicated subsets. For this the following notion, borrowed from \cite{Khovanskii1997}, will be convenient.

\begin{defin}\label{def:partition}
	A \emph{partition} of a polytope $P\subseteq V$ is a finite set $\Pi$ of polytopes in $V$ such that
	\begin{enumerate}
		\item $\bigcup_{Q\in\Pi} Q  = P$;
		\item\label{def:partition2} If $Q\in\Pi$ and $F\subseteq Q$ is a face, then $F\in\Pi$;
		\item\label{def:partition3} If $Q_1,Q_2\in\Pi$ and $Q_1\cap	Q_2\neq\emptyset$, then $Q_1\cap Q_2$ is a face in both $Q_1$ and $Q_2$.
	\end{enumerate}
	The elements of $\Pi$ that have the same dimension as $P$ are called the \emph{pieces} of $\Pi$. For notational convenience that will become clear in \cref{partition relation}, let 
	\[\Pi^\partial = \{Q\in\Pi\mid Q\not\subseteq \partial P\}.\]
\end{defin}

\begin{ex}\label{ex:partition}
	\begin{enumerate}
		\item\label{trivial partition} Given a polytope $P$, let $\F(P)$ denote the set of all faces of $P$ (including the codimension $0$ face $P$). Then $\F(P)$ is a partition of $P$.
		\item\label{partition from hyp} Let $P\subseteq V$ be a polytope and let $H_1,...,H_m\subseteq V$ be a collection of hyperplanes. Let $\Pi$ be the set that contains the closure of every connected component of $P\setminus\bigcup_{j=1}^m H_j$, together with all its faces. It is easy to see that $\Pi$ is indeed a partition of $P$, which we call the \emph{partition of $P$ with respect to $H_1,..., H_m$}. If $P\cap\bigcup_{j=1}^m H_j\subset \partial P$, then we obtain the trivial partition of part (\ref{trivial partition}) as a special case.
	\end{enumerate}
\end{ex}

The next proposition is an extension of \cref{cutting along a hyperplane} as well as a direct analogue of \cite[Proposition 3]{Khovanskii1997} although the proof is of entirely different nature.

\begin{prop}[Partition relation]\label{partition relation}
	Let $P\subseteq\R^n$ be a polytope and $\Pi$ be a partition of $P$. Then we have in $\P(\R^n)$ the equation
	\[ P = \sum_{Q\in\Pi^\partial} (-1)^{\codim (Q\subseteq P)} \cdot Q.\]
\end{prop}
\begin{proof}
	Since the statement does not depend on the ambient space, we may assume that $P$ is full-dimensional.
	
	We first deal with the case that $\Pi$ comes from a collection of hyperplanes $H_1,...,H_m\subseteq\R^n$ as in \cref{ex:partition} (\ref{partition from hyp}). We proceed by induction on $m$, where the base case $m=1$ is precisely Lemma \ref{cutting along a hyperplane}.
	
	For the induction step from $m-1$ to $m$, we denote the two halves of $P$ with respect to $H_m$ by $P_{\pm}$, and write $P_H = P\cap H_m$. We may assume that $P_{\pm}$ are codimension $0$ subpolytopes of $P$ since we could otherwise discard $H_m$ in the collection of hyperplanes without changing the induced partition of $P$. Further let $\Pi_+$ (resp. $\Pi_-$, $\Pi_H$) be the partition of $P_+$ (resp. $P_-$, $P_H$) with respect to $H_1, ..., H_{m-1}$. By the induction hypothesis we have 

	\begin{align}	\label{eq:partition}
	\begin{split}
		P &= P_+ + P_- - P_H\\
		P_\pm &= \sum_{Q\in\Pi_\pm^\partial} (-1)^{\codim  (Q\subseteq P_\pm)} \cdot Q \\
		P_H &= \sum_{Q\in\Pi_H^\partial} (-1)^{\codim (Q\subseteq P_H)} \cdot Q.
	\end{split}
	\end{align}
	
	Because of the boundary condition, we have a disjoint decomposition
	\[ \Pi^\partial = \Pi_+^\partial \cup  \Pi_-^\partial \cup  \Pi_H^\partial,\]
	which, together with (\ref{eq:partition}), implies the desired equation.
	
	For a general partition $\Pi$, let $\H$ be the set of those hyperplanes in $V$ which contain a $(\dim(V)-1)$-dimensional polytope of $\Pi$. Let $\Sigma$ be the partition of $P$ with respect to $\H$. We can think of $\Sigma$ as obtained from $\Pi$ by extending the facets of $\Pi$ through $P$, see \cref{fig:partition}.
	\begin{figure}[h]
		\captionsetup{width=0.7\textwidth}
		\centering
		\begin{tikzpicture}[line cap=round,line join=round,>=triangle 45,x=1.0cm,y=1.0cm, scale = 1]
		\clip(-1.1, -1.6) rectangle (5,3.2);
		\draw[thick]  (-1,-1) -- (-0.5, 1.2) -- (1, 2.5) -- (2.7, 2.7) -- (4.7, 0.3) -- (3, -1.5) -- (-1,-1);
		
		\draw[] (-0.5,1.2) -- (1, 0.5) -- (1.5, -1.3);
		\draw[] (1, 0.5) -- (4.35, 0.7);
		
		\draw[dashed, thick] (0.9, 0.49) -- (-0.65, 0.4);
		\draw[dashed, thick] (1.1, 0.45) -- (3.65, -0.75);
		\draw[dashed, thick]  (0.96, 0.6) -- (0.55, 2.05);
		
		\end{tikzpicture}
		\caption{If the straight lines indicate $\Pi$, then the straight and dashed lines together indicate $\Sigma$.}\label{fig:partition}
	\end{figure}
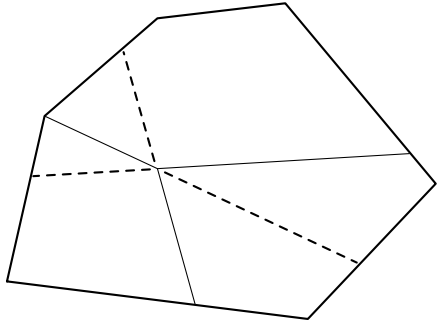
	
	For each $S\in\Pi$ let $\Sigma_S$ be the partition of $S$ with respect to $\H$. From the first part we know
	\begin{align}
		 P = \sum_{Q\in\Sigma^\partial} (-1)^{\codim (Q\subseteq P)} \cdot Q\label{eq1}\\
		 S = \sum_{Q\in\Sigma_S^\partial} (-1)^{\codim (Q\subseteq S)} \cdot Q.\label{eq2}
	\end{align}
	As in the first case, it is straightforward to check that there is a disjoint decomposition
	\begin{equation}\label{smaller partition}
		\Sigma^\partial = \coprod_{S\in\Pi^\partial} \Sigma_S^\partial.
	\end{equation}
	
	Now (\ref{eq1}), (\ref{eq2}), and (\ref{smaller partition}) give
	\begin{align*}
	P &= \sum_{Q\in\Sigma^\partial} (-1)^{\codim (Q\subseteq P)} \cdot Q\\
		 &= \sum_{S\in\Pi^\partial}\sum_{Q\in\Sigma_S^\partial} (-1)^{\codim (Q\subseteq P)} \cdot Q\\
		&= \sum_{S\in\Pi^\partial}\sum_{Q\in\Sigma_S^\partial} (-1)^{\codim (Q\subseteq S) + \codim(S\subseteq P)} \cdot Q\\
		&= \sum_{S\in\Pi^\partial} (-1)^{\codim (S\subseteq P)} \cdot S.
	\end{align*}
\end{proof}

\begin{rmk}
	If $P$ and all of the polytopes in $\Pi$ are integral, then the final step in the previous proof, i.e., the reduction to a partition coming from a collection of hyperplanes produces possibly non-integral polytopes. Nevertheless, the final equation contains only elements in the subgroup $\P(\Z^n)\subseteq\P(\R^n)$.
\end{rmk}

\subsection{The shadow partition}

Recall that a polytope in $\R^n$ is horizontal if it lies in a translate of the hyperplane $\zz^\perp$.

\begin{defin}
	Let $P\subseteq\R^n$ be a polytope of dimension $n$.
	\begin{enumerate}
		\item $P$ is \emph{grounded} if it has only one bottom face and this bottom face is horizontal. This unique bottom face will be called the \emph{ground}.
		\item $P$ is a \emph{pillar} if there is a horizontal polytope $Q$ and a $k> 0$ such that $P = Q+k\cdot\ZZ$.
		\item $P$ is an \emph{almost-pillar} if it has a unique bottom face and a unique top face.
	\end{enumerate}
\end{defin}

We record the following properties and leave their proofs as an easy exercise.

\begin{lem}\label{pillar}
	\begin{enumerate}
		\item Let $P\subseteq\R^n$ be a grounded polytope of dimension $n$ whose ground $G$ is contained in the hyperplane $H = \{ x\in\R^n\mid x_n = h\}$. Then the image of the \emph{grounding map}
		\[ g\colon P\to \R^n, \; (x_1,...,x_n)\mapsto (x_1,...,x_{n-1}, h)\]
		is $G$.
		\item Every pillar is an almost-pillar.
		\item\label{shadow grounded} For any polytope $P\subseteq \R^n$ such that $\dim(\Sh(P)) = n$, $\Sh(P)$ is grounded.
		\item If $P\subseteq\R^n$ is contained in a hyperplane which is not horizontal and \linebreak $\dim(\Sh(P)) = n$, then $\Sh(P)$ is a grounded almost-pillar.
	\end{enumerate}
\end{lem}

\begin{lem}\label{almost-pillars to pillars}
	Let $P$ be a polytope such that $*P$ is a grounded almost-pillar. Let $F$ be the unique bottom face of $P$. Then there exists a pillar $Q$ and a grounded almost-pillar $S$ such that in $\P(\R^n)$ we have
	\[ P = Q + F - S.\]
\end{lem}
\begin{proof}
	The easy case where $P$ is a pillar is left to the reader. If $P$ is not a pillar, define $S =  \Sh(F)$. By the previous lemma $S$ is a grounded almost-pillar. The union $Q = S\cup P$ is a pillar, and cutting $Q$ along $F = P\cap S$ produces the relation
	\[ Q = P + S - F\]
	by the cutting relation (see Lemma \ref{cutting along a hyperplane}).
\end{proof}

The following proposition will be one of the main tools for building a basis since it tells us how a polytope can be decomposed into smaller pieces in a controlled way.  We can then invoke the partition relation (see Proposition \ref{partition relation}) to turn this decomposition into a group-theoretic relation.

\begin{prop}[Shadow partition]\label{shadow partition}
	Let $P\subseteq\R^n$ be a grounded polytope. For every top face $F$ of $P$, let 
	\begin{equation}\label{shadow eq}
		P(F) = \Sh(F) + (h(F)-h(P))\cdot*\ZZ.
	\end{equation}
	Then the set
	\[ \Pi = \bigcup_{\substack{ F\subseteq P \\\text{ top face}}}\;\F(P(F))\] 
	is a partition of $P$ (see also Figure \ref{decomposition}) that will be referred to as the \emph{shadow partition} of $P$. If $P$ is integral, then the shadow partition contains only integral polytopes.
\end{prop}

\begin{center}
	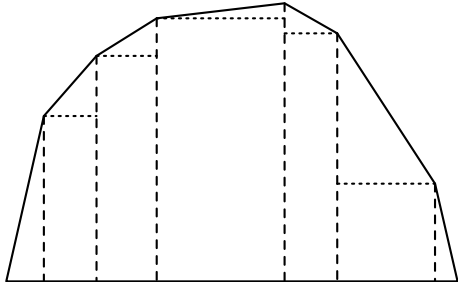
\begin{figure}[h]
		\captionsetup{width=0.7\textwidth}
		\begin{tikzpicture}[line cap=round,line join=round,>=triangle 45,x=1.0cm,y=1.0cm, scale = 1]
		\clip(-1.5, -1) rectangle (5,3.2);
		\draw[thick]  (5,-1) -- (-1,-1) -- (-0.5, 1.2) -- (0.2, 2) -- (1, 2.5) -- (2.7, 2.7) -- (3.4, 2.3) -- (4.7, 0.3) -- (5,-1);
		\draw[dotted, thick]  (-0.5,1.2) -- (0.2, 1.2);
		\draw[dotted, thick]  (0.2,2) -- (1, 2);
		\draw[dotted, thick]  (1,2.5) -- (2.7,2.5);
		\draw[dotted, thick] (2.7, 2.3) -- (3.4, 2.3);
		\draw[dotted, thick]  (3.4, 0.3) -- (4.7, 0.3);
		
		\draw[dashed, thick] (-0.5,1.2) -- (-0.5, -1);
		\draw[dashed, thick] (0.2,2) -- (0.2, -1);
		\draw[dashed, thick] (1,2.5) -- (1, -1);
		\draw[dashed, thick] (2.7,2.7) -- (2.7, -1);
		\draw[dashed, thick] (3.4,2.3) -- (3.4, -1);
		\draw[dashed, thick] (4.7,0.3) -- (4.7, -1);
		\end{tikzpicture}
		\caption{The dashed vertical lines indicate the shadow partition of a $2$-dimensional grounded polytope. Within each $P(F)$ as in (\ref{shadow eq}), the dotted horizontal line is the ground of $\Sh(F)$.}\label{decomposition}
	\end{figure}
\end{center}	

\begin{proof}
	The second condition on a partition, namely that faces of elements in $\Pi$ are themselves in $\Pi$ (see \cref{def:partition}), is clear.
		
	Next we prove $P = \bigcup_{Q\in\Pi} Q$. Let $g\colon P\to G$ be the grounding map. For any point $p\in P$ there exists a top face $F$ and $f\in F$ such that $g(p) =  g(f)$. Then $f$ and $g(p)$ are contained in $P(F)$. Since $p$ is a convex combination of $f$ and $g(p)$ and $P(F)$ is convex, we see $p\in P(F)$. Hence $P \subseteq \bigcup_{Q\in\Pi} Q$. For the reverse inclusion, we observe
	\[ h(P(F)) = h(\Sh(F)) - h(F) + h(P) = h(P)\]
	from which the inclusion $P(F) \subseteq P$ follows since $P$ is grounded.
	
	We finally need to show that for any $Q, Q'\in\Pi$ the intersection $Q\cap Q'$ is empty or a face in both of them. It suffices to do this for elements in $\Pi$ with maximal dimension. If $Q = P(F), Q' = P(F')$, then $Q\cap Q'=\emptyset$ if $F\cap F'= \emptyset$. Otherwise $F\cap F'$ is a face in both $F$ and $F'$, and so
	\[ Q\cap Q' = \Sh(F\cap F') + (h(F\cap F')-h(P))\cdot*\ZZ\]
	is a face in $Q$ and $Q'$.
	
	 If $P$ is integral, then for all top faces of $P$ the shadow $\Sh(F)$ is integral, and $h(F)$ and $h(P)$ are integers. Thus $P(F)$ is integral.
\end{proof}

\subsection{Vertical stretching}

This short section recalls \cite[Lemma 3.3]{ChaFriedlFunke2015} in a slightly more detailed form.

\begin{lem}[Vertical stretching]\label{vertical stretching}
	Let $H = \{x\in\R^n\mid x_n = h\}$ be a horizontal hyperplane. Then for every integral polytope $P\subseteq\R^n$ of dimension $n$, there exists an integer $k\geq 0$ such that for $Q = P + k\cdot(\ZZ + *\ZZ)$ we have:
	\begin{enumerate}
		\item $Q\cap H = c_h(Q)$ is an integral polytope;
		\item One half $Q_+$ of $Q$ with respect to $H$ is grounded;
		\item The other half $Q_-$ is such that $Q_-*$ is grounded.
	\end{enumerate}
\end{lem}
\begin{proof}
	The value $k = \max\{\left|p_n-h\right|\mid p\in P\}$ will do. The details can be found in \cite[Lemma 3.3]{ChaFriedlFunke2015}.
\end{proof}

\vspace{3mm}
\section{A basis for the integral polytope group}

\subsection{The subinduction step: Increasing the dimension of the polytopes}

In this section we construct an explicit basis for $\P(\Z^n)$, built from bases for the various subgroups of $\Z^n$. Roughly speaking, we throw together all these bases and their images under the shadow map. In order for this to work, we need a compatibility condition on these bases which ensures that they fit together. For this, we need the following notation.

\begin{defin}
	Let $H$ be a finitely generated free-abelian group. Given $1\leq m\leq n$, we denote by  $\P_T^m(H)$ the subgroup of $\P_T(H)$ generated by the polytopes of dimension at most $m$.
\end{defin}

Secondly, we want to avoid considering subgroups of the same rank nested inside each other.

\begin{defin}
	Let $H$ be a finitely generated free-abelian group. A subgroup $G\subseteq H$ is \emph{pure} if there is a subspace $U\subseteq V_H$ such that $G = U\cap H$.
\end{defin}

Note that a subgroup of $H$ is pure if and only if it is a direct summand of $H$.

\begin{prop}[Adding the last dimension]\label{add top dimension}
	Let $n\geq 2$. Assume that there are sets $\B_1\subseteq \B_2 \subseteq  ... \subseteq\B_{n-1}\subseteq\P_T(\Z^n)$ such that 
	\begin{enumerate}
		\item $\B_m\cap \P_T(G)$ is a basis for $\P^m_T(G)$ for every pure subgroup $G\subseteq\Z^n$ and $1\leq m\leq n-1$;
		\item $\B_m\setminus \B_{m-1}$ contains only polytopes of dimension $m$.
	\end{enumerate}
	Then there is a set $\C_n\subseteq\P_T(\Z^n)$ containing only polytopes of dimension $n$ such that $\B_{n-1}\cup\C_n$ is a basis for $\P_T(\Z^n)$.
\end{prop}

\begin{proof}
	Let
	\[ \C_n = \{\:\Sh(B) \mid B\in\B_{n-1}, \;\Sh(B) \text{ is $n$-dimensional }\}.\]

	We first prove that $\B_n := \B_{n-1}\cup\C_n$ is a generating set for $\P_T(\Z^n)$. Let $\langle S \rangle\subseteq\P_T(\Z^n)$ denote the subgroup generated by a subset $S$. 
	
	Let $P\subseteq\R^n$ be an integral polytope. We may assume without loss of generality that $P$ is $n$-dimensional, otherwise add the unit $n$-cube to it which we already know to lie in $\langle\B_1\rangle\subseteq\langle\B_n\rangle$ -- the $n$-cube is the Minkowski sum of horizontal and vertical unit segments.
	
	Note that in $\P_T(\Z^n)$ we have $\ZZ = *\ZZ$. By vertical stretching (see \cref{vertical stretching}), there is $k\in\Z$ and a horizontal hyperplane $H$ such that $P + k\cdot\ZZ$ intersects $H$ in an integral polytope $P'$ and cutting along this intersection produces a grounded half $P_+$ and a half $P_-$ such that $*P_-$ is grounded. By the cutting relation (see \cref{cutting along a hyperplane}), we have in $\P_T(\Z^n)$ 
	\[ P = P_+ + P_- - P' - k\cdot\ZZ,\]
	and since $P'$ and $\ZZ$ lie by assumption in $\langle\B_{n-1}\rangle$, it suffices to deal with $P_+$ and $P_-$ individually.
	
	First we take care of $P_+$. Note that $P_+$ is a grounded polytope with ground $P'$. Let $\Pi$ be the shadow partition of $P_+$ in the sense of \cref{shadow partition}. All polytopes in $\Pi$ of dimension at most $n-1$ lie in $\langle\B_{n-1}\rangle$. The remaining elements of $\Pi$ are of the form
	\[ 	P(F) = \Sh(F) + (h(F)-h(P))\cdot*\ZZ\]
	for some top face $F\subset P_+$. If we show that all the polytopes $P(F)$ lie in $\langle\B_n\rangle$, then the partition relation (see \cref{partition relation}) implies $P_+\in \langle\B_n\rangle$.
	
	By assumption there are $B_i\in \B_{n-1}$ and $\lambda_i\in\Z$ such that $F = \sum_{i=1}^k \lambda_i\cdot B_i$. By \cref{shadow lemma} the shadow map is a group homomorphism, so we have 
	\begin{equation}\label{shadow of face}
		\Sh(F) = \sum_{i=1}^k \lambda_i\cdot \Sh(B_i).
	\end{equation}
	If $\Sh(B_i)$ is $n$-dimensional, then $\Sh(B_i)\in\C_n\subseteq\B_n$, and otherwise $\Sh(B_i)\in\langle\B_{n-1}\rangle$. Hence it follows from (\ref{shadow of face}) that $\Sh(F)$ and therefore $P(F)$ lie in $\langle\B_n\rangle$.
	
	In order to deal with $P_-$, it suffices to show that $\langle\B_n\rangle$ is closed under the involution. Let $B\in\B_n$. Again by assumption, there is nothing to prove if $B\in\B_{n-1}$, so let $B\in\C_n$. Then $B$ is a grounded almost-pillar by \cref{pillar}. \cref{almost-pillars to pillars} applies to produce a pillar $Q$ and a grounded almost-pillar $S$ such that
	\[ *B = Q + *F - S,\]
	where $F$ is the top face of $B$. We have $Q, *F\in\langle\B_{n-1}\rangle$. Since $S$ is a grounded polytope, we may proceed with it as with $P_+$ to verify $S\in \langle\B_n\rangle$, and so $*B\in \langle\B_n\rangle$. This completes the proof that $\langle\B_n\rangle = \P_T(\Z^n)$.
	\medskip
	
	Next we prove that $\B_n$ is linearly independent. Assume that we have pairwise distinct elements $P^i_j\in\B_i\setminus\B_{i-1}$ and integers $\lambda_j^i\in\Z$ for $1\leq i\leq n$ and $1\leq j\leq s_i$ such that
	\begin{equation} \label{linear combination}
		\sum_{i=1}^n \sum_{j=1}^{s_i} \lambda_j^i\cdot P_j^i = 0.
	\end{equation}
	
	Since $\B_{n-1}$ is linearly independent, it suffices to show that $\lambda_k^n = 0$ for all $1\leq k\leq s_n$. For this we first need an auxiliary step.
	
	\medskip
	\emph{Claim:} If $P_k^{n-1}\in\B_{n-1}$ such that $\Sh(P_k^{n-1})$ is $n$-dimensional, then $\lambda_k^{n-1} = 0$.
	\smallskip
	
	Let $H\subseteq\R^n$ be the rational hyperplane containing $P_k^{n-1}$. Then $G = H\cap \Z^n$ is a pure subgroup. Since $\Sh(P_k^{n-1})$ is $n$-dimensional, $H$ is not horizontal, so there is $\phi\in\Hom(\R^n,\R)$ and $c\in\R$ with $H = \{x\in\R^n\mid \phi(x) = c\}$ and $\phi(\zz) < 0$. Since face maps are linear (see (\ref{face map})), applying the face map in $\phi$-direction to (\ref{linear combination}) yields the equation
	\begin{equation}\label{linear combination 2}
		\sum_{i=1}^n \sum_{j=1}^{s_i} \lambda_j^i\cdot F_\phi(P_j^i) = 0.
	\end{equation}
	in $\P_T(G)$. We claim that $F_\phi(P_j^i)$ has dimension $n-1$ if and only if $i = n-1$ and $P_j^i\in\P_T(G)$. The 'if'-part is obvious. The 'only if'-part is obvious except for the full-dimensional $P_j^n$, $1\leq j\leq s_n$. But since $P_j^n$ is grounded by Lemma \ref{pillar} (\ref{shadow grounded}) and $\phi(\zz) < 0$, we have $F_\phi(P_j^n) = F_\phi(A)$, where $A$ is the ground of $P_j^n$. Since $A$ is horizontal and $\phi(\zz) < 0$, $F_\phi(A)$ is a proper subface of $A$ and thus at most $(n-2)$-dimensional.
	
	This means that (\ref{linear combination 2}) breaks up into a sum $x$ of $(n-1)$-dimensional elements in $\B_{n-1}\cap\P_T(G)$ and a sum $y$ in $\P_T^{n-2}(G)$. Since the basis $\B_{n-1}\cap \P_T(G)$ of $\P_T^{n-1}(G)$ extends the basis $\B_{n-2}\cap \P_T(G)$ of $\P_T^{n-2}(G)$, this can only happen if $x = y = 0$. Hence $\lambda_j^{n-1} = 0$ for all $j$ such that $P_j^{n-1}\in\P_T(G)$ which includes in particular $j = k$. This proves the claim, which brings us to the original goal.
	
	\medskip
	\emph{Claim:} For all $1\leq k\leq s_n$ we have $\lambda_k^n = 0$.
	\smallskip
	
	Write $P_k^n = \Sh(B)$ for some $B\in\B_{n-1}$, and let $H$ be the affine rational hyperplane containing $B$. Take $\psi$ with $H = \{x\in\R^n\mid \psi(x) = c\}$ and $\psi(\zz) > 0$, and let $G = \Z^n\cap H$. Notice that then
	\[ F_\psi(P_k^n) = F_\psi(B) = B,\]
	but the previous claim ensures that $\lambda_k^{n-1} = 0$ if $P_k^{n-1} = B$. Thus the summands in 
	\begin{equation*}
		\sum_{i=1}^n \sum_{j=1}^{s_i} \lambda_j^i\cdot F_\psi(P_j^i) = 0
	\end{equation*}
	are distinct elements of $\B_{n-1}\cap\P_T(G)$ and elements lying in $\P_T^{n-2}(G)$. By the same argument as in the previous claim we deduce $\lambda_k^n = 0$.
\end{proof}

\subsection{The induction step: Increasing the rank}

We are now in a position to prove the main result of this paper, i.e., the existence of a geometrically tangible basis for the integral polytope group. 

\begin{constr}
We construct subsets $\B_1\subseteq \B_2 \subseteq  ... \subseteq\B_n\subseteq\P_T(\Z^n)$ by induction as follows.
For the base case, let $\B_1$ be the set of (translation classes of) $1$-dimensional polytopes in $\P_T(\Z^n)$ which are not a proper multiple of another (translation class of a) $1$-dimensional polytope in $\P_T(\Z^n)$. 

For the induction step from $m-1$ to $m$, we assume that the sets $\B_1\subseteq  ... \subseteq\B_{m-1}$ have been constructed. Consider the set
\[\U_{m} = \{U\subseteq \Z^n\mid U \text{ is a pure subgroup of rank } m\}.\]
Given some $U\in \U_{m}$, Proposition \ref{add top dimension} allows us to extend $\B_{m-1} \cap \P_T(U)$ to a basis $\B_m^U$ of $\P_T(U)$. Now put
\[ \B_m = \bigcup_{U\in\U_m} \B_m^U.\]
\end{constr}

\begin{thm}[Basis for the integral polytope group]\label{main theorem}
	The sets $\B_1\subseteq \B_2 \subseteq  ... \subseteq\B_n\subseteq\P_T(\Z^n)$ constructed above have the following properties:
		\begin{enumerate}
			\item $\B_m\cap \P_T(G)$ is a basis for $\P^m_T(G)$ for every pure subgroup $G\subseteq\Z^n$ and $1\leq m\leq n$;
			\item $\B_m\setminus \B_{m-1}$ contains only polytopes of dimension $m$.
		\end{enumerate}
	 In particular, $\B_n$ is a basis for $\P_T(\Z^n)$.
	 
	 Moreover, if $\mathcal{A}\subset \Z^n$ denotes a basis of $\Z^n$ and $\B_n'\subset \P(\Z^n)$ is a set of representatives for $\B_n\subseteq\P_T(\Z^n)$, then $\mathcal{A}\cup \B_n'$ is a basis for $\P(\Z^n)$.
\end{thm}

\begin{proof}
	The proof also proceeds by induction. Clearly, $\B_1\cap \P_T(G)$ is a generating set for $\P^1_T(G)$ provided that $G\subset \Z^n$ is a \emph{pure} subgroup. On the other hand, the edges of the Minkowski sum of pairwise non-parallel segments are translates of these segments. This readily implies that $\B_1$ is linearly independent.

	For the induction step it is clear that $\B_m\s- \B_{m-1}$ contains only polytopes of dimension $m$. We need to verify that $\B_m\cap \P_T(G)$ is a basis for $\P_T^m(G)$ for every pure subgroup $G\subseteq \Z^n$. For $\rank(G)\leq m-1$ this is obvious since $\B_m\cap \P_T(G) = \B_{m-1}\cap \P_T(G)$ and $\P_T^m(G) = \P_T^{m-1}(G)$. If $\rank(G) = m$, then $G\in \U_m$, and thus $\B_m\cap \P_T(G)  = \B_m^G$ and $\P_T(G) = \P_T^m(G)$. Finally, let $\rank(G) > m$ and consider the set
	\[\U_m^G = \{U\subseteq G\mid U \text{ is a pure subgroup of rank } m\}\subseteq\U_m.\]
	
	Then $\P_T^m(G)$ is generated by the union of all $\P_T(U)$ with $U\in \U_m^G$. On the other hand, each such $\P_T(U)$ is generated by $\B_m^U\subseteq \B_m$. This shows that $\B_m\cap \P_T(G)$ generates $\P_T^m(G)$. It remains to prove that $\B_m$ is linearly independent. This is in very much the same spirit as the corresponding proof of Proposition \ref{add top dimension}. 
	
	Let $P_i \in \B_m$ be pairwise distinct elements and $\lambda_i\in\Z$ ($1\leq i\leq k$) such that 
	\[\sum_{i=1}^s \lambda_i\cdot P_i = 0.\] 
	Again it suffices to prove $\lambda_i = 0$ for all $i$ such that $P_i\in\B_m\setminus\B_{m-1}$ since $\B_{m-1}$ is assumed to be linearly independent. For a fixed $P_j\in \B_m\setminus \B_{m-1}$, let $U\in\U_m$ be such that $P_j\in\P_T(U)$. Let $H\subseteq\R^n$ be a rational hyperplane such that 
	\begin{equation}\label{hyperplane}
		U\cap U' = H\cap U'
	\end{equation}
	for every $U'\in\U_m$ for which there exists an index $i$ with  $P_i\in\P_T(U')\cap (\B_m\setminus \B_{m-1})$. Pick $\phi\in\Hom(\R^n,\R)$ with $H=\ker\phi$. Applying the face map induces the equation
	\[\sum_{i=1}^s \lambda_i\cdot F_\phi(P_i) = 0.\]
	Because of (\ref{hyperplane}), $F_\phi(P_i)$ is $m$-dimensional if and only if $P_i$ is $m$-dimensional and $P_i\in\P_T(U)$, or in other words $P_i\in\B_m^U\setminus\B_{m-1}$, and the remaining summands lie in $\P_T^{m-1}(U)$. Since $\B_m^U$ extends the basis $\B_{m-1}\cap\P_T(U)$ of $\P_T^{m-1}(U)$, we see that $\lambda_i = 0$ for all $i$ such that $F_\phi(P_i)$ is $m$-dimensional. In particular $\lambda_j=0$ and the proof is complete.
	
	The 'moreover'-part follows directly from the split exactness of the sequence
	\[ 0\to H\to\P(H)\to\P_T(H)\to 0\]
	which was first proved in \cite[Lemma 3.8 (2)]{FriedlLueck2015b}, but follows now also from the fact that $\P_T(H)$ is free-abelian.
\end{proof}

\begin{ex}
	Going through the proof of Theorem \ref{main theorem}, we see that a basis for $\P_T(\Z^2)$ is given by the set comprising the $1$-dimensional polytopes which are not proper multiples of another integral polytope, and their shadows, which are rectangular triangles.
\end{ex}

\begin{rmk}
	The above construction applies also to produce a basis for the real vector space $\P(\R^n)$. The only wording that needs to be replaced is \emph{pure subgroup} with \emph{subspace}.
\end{rmk}

\vspace{3mm}
\section{The involution as face Euler characteristic}

In this section we identify the involution on the polytope group with the following object. Recall that $\F(P)$ denotes the set of faces of a polytope $P$, including $P$ itself.

\begin{defin}[Face Euler characteristic]
	Given a polytope $P\subseteq V$, we call 
	\[ \chi_\F(P) = \sum_{F\in\F(P)} (-1)^{\dim(F)}\cdot F \in\P(V)\]
	the \emph{face Euler characteristic} of $P$.
\end{defin}

It is at this point not clear that $\chi_\F$ is additive on the Minkowski sum, so we do not immediately obtain a map on all of $\P(V)$. However, in this section we will prove the following.

\begin{thm}[Involution as face Euler characteristic]\label{main theorem 2}
	For any polytope $P\subseteq V$ we have in $\P(V)$
	\[ *P = - \chi_\F(P).\]
\end{thm}

 This theorem is inspired by \cite[Theorem 2]{KhovanskiiPukhlikov1992} and \cite[Theorem 2]{McMullen1989}. The latter result takes formally precisely the same form as ours, but it is there an equation in a so-called \emph{polytope algebra} that carries the Minkowski sum as \emph{multiplication}, and what we call cutting relation (see \cref{cutting along a hyperplane}) as addition. We emphasize that \cref{main theorem 2} restricts also to the integral polytope group since the faces of an integral polytope are integral.

 The following corollary can be seen as a combinatorial reminiscence of the fact that the Euler characteristic of a closed odd-dimensional manifold vanishes and the Euler characteristic of a closed even-dimensional manifold which bounds a compact manifold is even.

\begin{cor}
	Let $P\subseteq V$ be a symmetric polytope. Then we have in $\P(V)$
	\[  
	\sum_{\substack{F\in\F(P) \\ \:F\neq P}} (-1)^{\dim(F)}\cdot F = 
	\begin{cases}
	0, &         \text{if } \dim(P) \text{ is odd};\\
	-2\cdot P, &         \text{if } \dim(P) \text{ is even}.
	\end{cases}	\]
\end{cor}

The second corollary does not seem to be trivial right from the definitions either.

\begin{cor}
	Given polytopes $P,Q\subseteq V$ we have 
	\[\chi_\F(P+Q) = \chi_\F(P) + \chi_\F(Q).\]
\end{cor}
\medskip

The strategy for the proof of \cref{main theorem} serves as a road map for proving \cref{main theorem 2}: We show a partition relation for face Euler characteristics, prove the statement for shadows, and combine these two facts to obtain the claim for any grounded polytope. The general case follows easily from this special case.

\begin{lem}\label{cutting hyperplane for face euler}
	Let $P\subseteq\R^n$ be a polytope and $H\subseteq\R^n$ be a hyperplane. Denote the two halves of $P$ with respect to $H$ by $P_+$ and $P_-$. Then 
	\begin{equation}\label{euler char sum}
		\chi_\F(P) +  \chi_\F(P\cap H) = \chi_\F(P_+) + \chi_\F(P_-).
	\end{equation}
	If \cref{main theorem 2} holds for any three of the polytopes $P, P_+, P_-, P\cap H$, then it also holds for the fourth.
\end{lem}
\begin{proof}
	We distinguish four cases as to how $H$ cuts a face $F\in\F(P)$:
	\begin{enumerate}
		\item If $F\cap H = \emptyset$, then $F$ is a face of one of the $P_i$ and contributes $(-1)^{\dim(F)}\cdot F$ to both sides of (\ref{euler char sum}).
		\item If $F\cap H = F$, then $F$ is a face of $P_+, P_-$ and $P\cap H$, and it contributes $(-1)^{\dim(F)}\cdot 2 F$ to both sides.
		\item If $F\cap H\neq F$ and $F\cap H$ is a face of $F$, then $F$ is a face of exactly one $P_{\pm}$ and contributes $(-1)^{\dim(F)}\cdot F$ to both sides. (Note that $F\cap H$ will itself then fall into case (2).)
		\item Otherwise, the cutting relation (see \cref{cutting along a hyperplane}) yields 
		\begin{equation}\label{eq34}
			F + (F\cap H) = F_+ +  F_-
		\end{equation} 
		for the two halves of $F$ with respect to $H$. But $F\cap H$ is also a face in $P_+, P_-$ and $P\cap H$ which is not covered by the other cases. This means that $F$ contributes $(-1)^{\dim(F)}\cdot(F - (F\cap H))$ to the left-hand side and $(-1)^{\dim(F)}\cdot(F_+ + F_- -2\cdot (F\cap H))$ to the right-hand side. These two values coincide by (\ref{eq34}).
	\end{enumerate} 
	Every summand of the face Euler characteristics has now been accounted for exactly once, so that (\ref{euler char sum}) follows. The last statement follows from comparing this with $*P + *(P\cap H) = *P_+ + *P_-$.
\end{proof}

\begin{prop}[Partition relation for face Euler characteristics]\label{partition euler face}
	Let $P\subseteq\R^n$ be a polytope and $\Pi$ be a partition of $P$. Then we have in $\P(\R^n)$ the equation
	\[ \chi_\F(P) = \sum_{Q\in\Pi^\partial} (-1)^{\codim (Q\subseteq P)} \cdot \chi_\F(Q).\]
	In particular, if \cref{main theorem 2} holds for all elements in $\Pi$, then it also holds for $P$.
\end{prop}
\begin{proof}
	This follows from \cref{cutting hyperplane for face euler} in exactly the same way as \cref{partition relation} (the partition relation) follows from \cref{cutting along a hyperplane} (the cutting relation).
\end{proof}

Next we show that \cref{main theorem 2} is true for the pieces in a shadow partition (see \cref{shadow partition}).

\begin{lem}[Face Euler characteristics of shadows]\label{face euler char of shadows}
	Assume that \cref{main theorem 2} is known for polytopes of dimension at most $n-1$. 
	\begin{enumerate}
		\item\label{euler1} If $P\subseteq\R^n$ is a polytope of dimension at most $n-1$, then we have
		\[ *\Sh(P) = -\chi_\F(\Sh(P));\]
		\item\label{euler2} If $P\subseteq\R^n$ is a polytope of dimension at most $n-1$, then we have \[ *(P+\ZZ) = -\chi_\F(P+\ZZ);\]
		\item\label{euler3} Let $P\subseteq\R^n$ be a polytope. Then \cref{main theorem 2} holds for $P$ if and only if it holds for $P + \ZZ$ (or equivalently $P+*\ZZ$).
	\end{enumerate}
\end{lem}
\begin{proof}
	(\ref{euler1}) We may assume that $\Sh(P)$ is of dimension $n$. Recall that $\Sh(P)$ is grounded by \cref{pillar}. Let $G\subset \Sh(P)$ be its ground and $g:\Sh(P) \to G$ be the grounding map. Every face $F\subset P$ such that $F \neq g(F)$ induces the following faces of $\Sh(P)$:
	\begin{enumerate}[label=(\roman*)]
		\item $F$ itself;
		\item $g(F)$, which has the same dimension as $F$;
		\item The intermediate face $\Sh(F) +  (h(F) - h(P))\cdot *\ZZ$, which has dimension $\dim(F) + 1$. Alternatively, this face equals $\hull(F \cup g(F))$.
	\end{enumerate}
	
	If $F = g(F)$, then we have the equality
	\[ F +g(F) - (\Sh(F) +  (h(F) - h(P))\cdot *\ZZ) = F+ F-F = F.\] 
	Hence we may as well take the three summands above instead of $F$ in the following calculations. In this way we avoid a case analysis and notational overload.
	
	The subsets of $\F(\Sh(P))$ corresponding to faces of type (i) and (ii) are $\F(P)$ and $\F(G)$, respectively. By assumption we have $*P = -\chi_\F(P)$ and $*G = -\chi_\F(G)$ since these are polytopes of dimension $n-1$. Now we calculate using the additivity of the shadow map (see \cref{shadow lemma}) 
	\begin{align}\label{eq:12}
	\begin{split}
	&\;\chi_\F(\Sh(P)) \\
	=&\; \sum_{F\in\F(\Sh(P))} (-1)^{\dim(F)}\cdot F \\
	=&\; \sum_{F\in\F(P)} (-1)^{\dim(F)}\cdot F  + \sum_{F\in\F(G)} (-1)^{\dim(F)}\cdot F \\ +&\; \sum_{F\in\F(P)} (-1)^{\dim(F)+1}\cdot (\Sh(F) +(h(F) - h(P))\cdot *\ZZ)\\ 
	=& \;\chi_\F(P) + \chi_\F(G) - \Sh(\chi_\F(P)) - \sum_{F\in\F(P)} (-1)^{\dim(F)}\cdot (h(F) - h(P))\cdot *\ZZ\\
	=& \; -*P - *G+\Sh(*P) -  h\cdot *\ZZ,
	\end{split}
	\end{align}
	where we put
	\[h = \sum_{F\in\F(P)} (-1)^{\dim(F)}\cdot (h(F) - h(P)).\] 
	Note that since the faces of $P$ determine a cell structure on $P$, we have 
	\[\sum_{F\in\F(P)} (-1)^{\dim(F)} = \chi(P) = 1\]
	and hence
	\begin{align}\label{eq:37}
	\begin{split}
	h =&\sum_{F\in\F(P)} (-1)^{\dim(F)}\cdot(h(F) - h(P)) \\
	=& -  h(P)\cdot \sum_{F\in\F(P)} (-1)^{\dim(F)} + \sum_{F\in\F(P)} (-1)^{\dim(F)}\cdot h(F) \\
	=& \;- h(P) + h(\chi_\F(P)) \\
	=& \; - h(P)-h(*P).
	\end{split}
	\end{align}
	
	Recall from \cref{rem:upper maps} that we may define a height and shadow map in the opposite direction
	\[ h^+\colon \P(\R^n)\to\R^n\;\;\text{ and }\;\;\Sh^+\colon \P(\R^n)\to \P(\R^n)\]
	satisfying the equations
	\begin{equation}\label{eq:face euler for shadows 12}
		h^+(*P) = -h(P)\;\;\text{ and }\;\; \Sh^+(*P) = *\Sh(P).
	\end{equation}
	
	Now consider the pillar 
	\[\Sh(*P)\cup\Sh^+(*P) = *G +  (h^+(P)-h(P))\cdot*\ZZ.\] 
	By (\ref{eq:37}) and (\ref{eq:face euler for shadows 12}) we have
	\[\Sh(*P) \cup *\Sh(P) = \Sh(*P)\cup\Sh^+(*P) = *G +  (h^+(P)-h(P))\cdot*\ZZ = *G +h\cdot*\ZZ.\]   
	By the cutting relation (see \cref{cutting along a hyperplane}), cutting this pillar along $*P$ gives
	\[  \Sh(*P) + *\Sh(P) = *G + h\cdot*\ZZ + *P. \]
	We conclude by comparing this with equation (\ref{eq:12})
	\[ *\Sh(P) = *G + h\cdot*\ZZ + *P-\Sh(*P) = -\chi_\F(\Sh(P)).\]
	
	(\ref{euler2}) This part is similar to the first one. The face analysis, which we leave to the reader, yields in this case
	\begin{align*}
	\chi_\F(P+\ZZ) &= \chi_\F(P) + \chi_\F(P+\zz) + \sum_{F\in\F(P)}(-1)^{\dim(F)+1}\cdot (F+\ZZ)\\
	&= 2\cdot\chi_\F(P) +\chi(P)\cdot\zz -\chi_\F(P) - \chi(P)\cdot\ZZ\\
	&= \chi_\F(P) +\zz-\ZZ\\
	&= -*P -*\ZZ\\
	&= -*(P+\ZZ).
	\end{align*}
	
	(\ref{euler3}) Assume that $\dim(P) = n$. There is a partition $\Pi$ of $P+\ZZ$ that has the pieces $P$ and $F+\ZZ$ for all top faces $F\subseteq P$, see \cref{partition of P+Z}.
	\begin{center}
		\begin{figure}[h]
			\captionsetup{width=0.7\textwidth}
			\begin{tikzpicture}[line cap=round,line join=round,>=triangle 45,x=1.0cm,y=1.0cm, scale = 1]
			\clip(-1.5, -2) rectangle (5.2,4.2);
			\draw[]  (5,-2) -- (-1,-1) -- (-0.5, 1.2) -- (0.2, 2) -- (1, 2.5) -- (2.7, 2.7) -- (3.4, 2.3) -- (4.7, 0.3) -- (5,-2);
	
			\draw[thick]  (5,-2) -- (-1,-1) -- (-1,0) -- (-0.5, 2.2) -- (0.2, 3) -- (1, 3.5) -- (2.7, 3.7) -- (3.4, 3.3) -- (4.7, 1.3) -- (5,-1) -- (5,-2);
			
			\draw[dashed]  (-0.5, 2.2) --(-0.5, 1.2); 
			\draw[dashed] (0.2,2) -- (0.2, 3); 
			\draw[dashed] (1,2.5) -- (1, 3.5);
			\draw[dashed] (2.7, 2.7) -- (2.7, 3.7);
			\draw[dashed] (3.4, 2.3) -- (3.4, 3.3);
			\draw[dashed] (4.7, 0.3)-- (4.7, 1.3); 
	
			\end{tikzpicture}
			\caption{A partition of $P+\ZZ$ with pieces $P$ and $F+\ZZ$ for all top faces $F\subseteq P$.}\label{partition of P+Z}
		\end{figure}
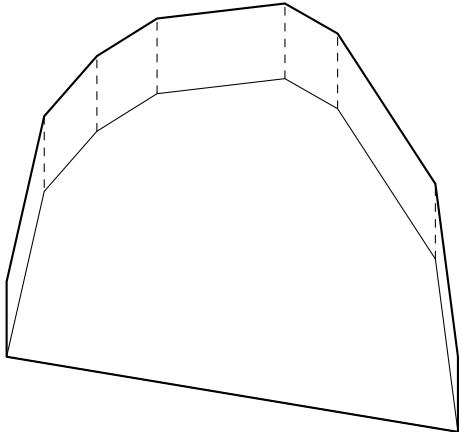
	\end{center}	
	
	By part (\ref{euler2}) and the assumption, \cref{main theorem 2} holds for all elements of this partition except possibly for $P$. Thus comparing the partition relation of polytopes (see \cref{partition relation})
	\[*(P +\ZZ)= *P + \sum_{Q\in\Pi^\partial, Q\neq P} (-1)^{\codim (Q\subseteq P)} \cdot *Q\] 
	with the partition relation for face Euler characteristics (see \cref{partition euler face})
	\[ \chi_\F(P+\ZZ) = \chi_\F(P) + \sum_{Q\in\Pi^\partial, Q\neq P} (-1)^{\codim (Q\subseteq P)} \cdot \chi_\F(Q)\]
	implies
	\[ *P = -\chi_\F(P) \;\;\text{ if and only if }\;\;  *(P+\ZZ) = -\chi_\F(P+\ZZ).\]
\end{proof}

For completeness we record the following trivial observation.

\begin{lem}\label{inversion face euler}
	For any polytope $Q$ we have $\F(*Q) = *\F(Q)$ and \cref{main theorem 2} is true for $Q$ if and only if it is true for $*Q$.
\end{lem}

We are now ready to prove the main theorem of this section.

\begin{proof}[Proof of Theorem \ref{main theorem 2}]
	Let $P$ be an arbitrary polytope. We prove the claim by induction on the dimension of $P$. If $\dim(P) = 0$, then there is nothing to prove.
	
	Let now $n=\dim(P)$. By vertical stretching (\cref{vertical stretching}) and \cref{face euler char of shadows} (\ref{euler3})  we may assume that $P$ can be cut along a horizotal polytope into a grounded half $P_+$ and a half $P_-$ such that $*(P_-)$ is grounded. By the cutting relation for face Euler characteristics (\cref{cutting hyperplane for face euler}), \cref{inversion face euler}, and the induction hypothesis, it suffices to prove the claim in the special case that $P$ is grounded.
	
	We consider the shadow partition $\Pi$ of $P$ (see Proposition \ref{shadow partition}). \cref{main theorem 2} is true for all elements in $\Pi$ by the induction hypothesis and \cref{face euler char of shadows} (\ref{euler1}) and (\ref{euler3}). The two partition relations of \cref{partition relation} and \cref{partition euler face} then imply it for $P$.
\end{proof}

\vspace{3mm}
\section{Polytopes with the same seminorm}

The main result of \cite{ChaFriedlFunke2015} states that 
\[\ker \big(\id-*\colon \P(H)\to \P(H)\big) = \im\big(\id + *\colon \P(H)\to \P(H)\big).\]
In this section we prove the dual result and put it in context with the following seminorm map on the polytope group.

The set of (set-theoretic) maps $\Map(\Hom_\Z(H, \R), \R)$ is a group under pointwise addition. An integral polytope in $P\subseteq H\otimes\R$ induces a seminorm on $\Hom_\Z(H,\R)$ by setting
\[ \|\phi\|_P = \max\{\phi(p)-\phi(q)\mid p,q\in P\}. \]
It is easy to verify $\|\phi\|_{P+Q} = \|\phi\|_P + \|\phi\|_Q$ which allows us to make the following definition.

\begin{dfn}[Seminorm homomorphism]\label{def:seminorm map}
	We call
	\[ \norm\colon\P(H)\to \Map(\Hom(H,\R),\R),\;\; P-Q\mapsto \|\cdot \|_P - \|\cdot\|_Q\]
	\emph{seminorm homomorphism}. It passes to the quotient $\P_T(H)$ and the induced map
	\[\norm\colon\P_T(H)\to \Map(\Hom(H,\R),\R)\]
	is denoted by the same symbol.
\end{dfn} 

\begin{remark}\label{lem:equivalence norm}
	Let $H$ be a finitely generated free-abelian group. Then we have
	\begin{align*}
		\ker\big(\norm\colon  \P(H)\to \Map(\Hom(H,\R),\R)\big)  &= \ker \big(\id+*\colon \P(H)\to \P(H)\big);\\
		\ker\big(\norm\colon  \P_T(H)\to \Map(\Hom(H,\R),\R)\big) &= \ker \big(\id+*\colon \P_T(H)\to \P_T(H)\big).
	\end{align*}
	Namely, it is shown in \cite[Section 3.7]{FriedlLueck2015b} that two integral polytopes $P$ and $Q$ satisfy $P+*P = Q +*Q$ if and only if $\|\cdot \|_P = \|\cdot\|_Q$.
\end{remark}

Clearly this common kernel contains $\im\big(\id-*:\P(H)\to\P(H)\big)$ (respectively $\im\big(\id-*:\P_T(H)\to\P_T(H)\big)$), and in the remainder of this section we prove that this is indeed an equality. In \cite{ChaFriedlFunke2015} the strategy to prove the inclusion $\ker(\id -*) \subseteq \im(\id+*)$ was to stretch a hyperplane in the $\zz$-direction and then cut it along $\zz^\perp$ to obtain two halves which are involutions of each other. For the dual inclusion $\ker(\id +*) \subseteq \im(\id-*)$ we need to glue halves of different polytopes together. In order to ensure that this gluing process produces a polytope, we need the following lemma. Recall the compression maps $c_h\colon \R^n\to\R^n$ from (\ref{compression}).

\begin{lem}[Vertical gluing]\label{lem:vertical gluing}
	Let $H = \{x\in\R^n\mid x_n = h\}$ be a horizontal hyperplane. If $P, Q\subseteq\R^n$ are two (integral) polytopes such that 
	\begin{equation}\label{gluing assumption}
	P\cap H = c_h(P) = c_h(Q) = Q\cap H, 
	\end{equation} 
	then the set $P_+\cup Q_-$ is a (integral) polytope, where $P_+$ denotes the upper half of $P$ and $Q_-$ denotes the lower half of $Q$ with respect to $H$. 
	
	If additionally $h=0$, i.e. $H=\zz^\perp$, then we have:
	\begin{enumerate}
		\item\label{item1a} $(P+*P)\cap H = (P\cap H) + (*P\cap H)$;
		\item\label{item2a}  $(P+*P)_+ = P_+ + *(P_-)$;
		\item\label{item3a} $(P+*P)_- = P_- + *(P_+)$.
	\end{enumerate}
\end{lem}
\begin{proof}
	Denote the vertex sets of $P_+$ resp. $Q_-$ by $V(P_+)$ resp. $V(Q_-)$. We will show
	\[ P_+\cup Q_-  = \hull( V(P_+)\cup V(Q_-)),\]
	where the inclusion $\subset$ is obvious.
	
	For the reverse inclusion, it suffices to show that $P_+\cup Q_-$ is convex. Let $p\in P_+$ and $q\in Q_-$, and take a convex combination $x = t\cdot p + (1-t)\cdot q$. Since $P_+$ and $Q_-$ are convex, we may assume that $x\in H$ (and deal with other convex combinations inside $P_+$ and $Q_-$ individually). We can also write $x = t\cdot c_h(p) + (1-t)\cdot c_h(q)$. Assumption (\ref{gluing assumption}) then implies that $x\in P\cap H = Q\cap H \subset P_+\cup Q_-$. This finishes the proof of the first statement.
	
	\smallskip
	In the equalities (\ref{item1a}), (\ref{item2a}), and (\ref{item3a}), the inclusion $\supseteq$ is true irrespective of the assumption that $P\cap H = c_h(P)$. 
	
	\smallskip
	To prove $\subseteq$ in (\ref{item1a}), let $p\in P,\: q\in *P$ with $p_n+q_n = 0$. Then $p+q = c_0(p+q) = c_0(p)+c_0(q)$ which lies in $(P\cap H) + (*P\cap H)$ since by assumption $c_0(P) = P\cap H$ and thus $c_0(*P) = *c_0(P) = *(P\cap H) = *P\cap H$. 
	\smallskip
	
	To prove $\subseteq$ in (\ref{item2a}), let $p\in P,\: q\in *P$ with $p_n+q_n \geq 0$. If $p_n, q_n\geq 0$, then $p\in P_+$ and $q\in *(P_-)$ and we are done. If $p_n\geq 0$ and $q_n\leq 0$, then take $p' = p  + q_n\cdot\zz$ and $q' = q -q_n\cdot\zz = c_0(q)$. We have $p'\in P_+$ since it is a convex combination of $p$ and $c_0(p)\in P$, and we have $q'\in *(P_-)$ since $c_0(*P)  = *P\cap H\subset *(P_-)$. Thus $p+q = p'+q' \in P_+ + *(P_-)$.
	
	The third claim is proved similarly.
\end{proof}

\begin{thm}\label{main theorem 3}
	We have
	\[\ker \big(\id+*\colon \P(H)\to \P(H)\big) = \im\big(\id - *\colon \P(H)\to \P(H)\big)\]
	and 
	\[\ker \big(\id+*\colon \P_T(H)\to \P_T(H)\big) 
	=  \im\big(\id-*\colon \P_T(H)\to\P_T(H)\big).\]
\end{thm}

\begin{proof}
	We deal with $\P(H)$ first. Again the inclusion $\supseteq$ in the claim $\ker(\id +*) = \im(\id-*)$ is obvious. For the opposite inclusion, we proceed again by induction on the rank of $H\cong \Z^n$. If $n=0$, then there is once more nothing to prove.
	
	Let $P-Q \in \ker(\id+ *)$, so
	\begin{equation}\label{same norm}
	P+*P = Q+*Q.
	\end{equation}
	After vertical stretching (see \cref{vertical stretching}), we may assume 
	\begin{equation}\label{cutting possible}
	P\cap H= c_0(P) \;\;\text{ and }\;\; Q\cap H = c_0(Q),
	\end{equation} 
	where here and henceforth we let $H = \zz^\perp$. Then \cref{lem:vertical gluing} (\ref{item1a}) together with (\ref{same norm}) implies
	\[	 (P\cap H) + (*P\cap H) = (Q\cap H) + (*Q\cap H).\]
	We may therefore apply the induction hypothesis to $(P\cap H) - (Q\cap H)$ and obtain an integral polytope $R$ contained in $H$ such that 
	\begin{equation}\label{cutting at H}
	(P\cap H) + *R = (Q\cap H) + R.
	\end{equation}
	Clearly $P+*R - (Q+R)\in\ker(\id+ *)$, and it suffices to prove that this element lies in $\im(\id-*)$. To ease notation, put $A = P+*R$ and $B = Q+R$. We see from (\ref{cutting possible}), (\ref{cutting at H}) and the fact that $R$ lies in $H$ the equalities
	\[G := c_0(A) = A\cap H =  (P\cap H) + *R = (Q\cap H) + R = B\cap H = c_0(B).\]
	We are therefore in the situation of \cref{lem:vertical gluing} so that the two halves $A_+$ and $B_-$ (with respect to $H$) can be glued together to give a polytope $S = A_+\cup B_-$. Moreover, \cref{lem:vertical gluing} (\ref{item3a}) gives
	\begin{equation}\label{implication}
	A_- + *(A_+) = (A+*A)_- = (B+*B)_- = B_- + *(B_+).
	\end{equation} 
	
	If we put $T = S-B$, then several applications of the cutting relation (see \cref{cutting along a hyperplane}) yield
	\begin{align*}
	T - *T &= S - *S - B + *B \\
	&= (A_+ + B_- - G) - (*A_+ + *B_- - *G) - (B_+ +B_- -G) + (*B_+ +*B_- -*G)\\
	&= A_+ + B_- + *B_+ - *A_+ - B_+ - B_-\\
	& \hspace{-2.3mm}\overset{(\ref{implication})}{=} A_+ + A_- + *A_+ - *A_+ - B_+ - B_-\\
	&= (A_+ + A_- -G) - (B_+ + B_- -G)\\
	&= A-B,
	\end{align*}	
	which completes the proof for $\P(H)$.
	
	\medskip
	We deduce the statement for the quotient $\P_T(H)$ as follows. The map
	\[ \sym\colon \P_T(H)\to\P(H), \;\; P-Q\mapsto P+*P - (Q+*Q)\]
	is well-defined and fits into the commutative diagram
	\[ \xymatrix{
		\P(H) \ar[r]^{\id +*} \ar[d] & \P(H) \ar[d]\\
		\P_T(H) \ar[r]^{\id+*} \ar[ur]^{\sym} & \P_T(H),
	}\]
	where the vertical maps are the projections. Since $\sym(x)$ is a difference of two  polytopes which are symmetric about the origin, $\sym(x)$ is a point if and only if it is zero. This implies 
	\[ \ker \big(\id+*\colon \P_T(H)\to \P_T(H)\big) 
	= \ker \big(\sym\colon \P_T(H)\to \P(H)\big). \]
	
	Because of the commutative diagram above, any preimage of an element $x\in \ker \big(\sym\colon \P_T(H)\to \P(H)\big)$ in $\P(H)$ lies in 
	\[\ker \big(\id+*\colon \P(H)\to \P(H)\big) = \im \big(\id-*\colon \P(H)\to \P(H)\big).\] 
	Thus 
	\[ \ker \big(\sym\colon \P_T(H)\to \P(H)\big) \subset \im \big(\id-*\colon \P_T(H)\to \P_T(H)\big)\]
	and the reverse inclusion is obvious. This finishes the proof of \cref{main theorem 3}.
\end{proof}

\begin{remark}
	It is in contrast to the previous theorem \emph{not} true that 
	\[ \ker \big(\id-*\colon \P_T(H)\to \P_T(H)\big) = \im \big(\id+*\colon \P_T(H)\to \P_T(H)\big)\]
	as can easily seen for $H=\Z$. 
\end{remark}
\medskip

\bibliography{bibliography}

\end{document}